\documentclass[11pt]{amsart}
\usepackage[margin=1.35in]{geometry}
\usepackage{amscd,amsmath,amsxtra,amsthm,amssymb,stmaryrd,xr,mathrsfs,mathtools,enumerate,commath, comment}
\usepackage{stmaryrd}
\usepackage{xcolor}
\usepackage{commath}
\usepackage{comment}
\usepackage{tikz-cd}
\usepackage{longtable} % for 'longtable' environment
\usepackage{pdflscape} % for 'landscape' environment
\usepackage{booktabs}
\usepackage{hyperref}
\definecolor{vegasgold}{rgb}{0.77, 0.7, 0.35}
\definecolor{darkgoldenrod}{rgb}{0.72, 0.53, 0.04}
\definecolor{gold(metallic)}{rgb}{0.83, 0.69, 0.22}
\hypersetup{
 colorlinks=true,
 linkcolor=darkgoldenrod,
 filecolor=brown,      
 urlcolor=gold(metallic),
 citecolor=darkgoldenrod,
 pdftitle={On the corank of the fine Selmer group of an elliptic curve over a $\mathbb{Z}_p$-extension},
 }

\usepackage[all,cmtip]{xy}

\DeclareFontFamily{U}{wncy}{}
\DeclareFontShape{U}{wncy}{m}{n}{<->wncyr10}{}
\DeclareSymbolFont{mcy}{U}{wncy}{m}{n}
\DeclareMathSymbol{\Sh}{\mathord}{mcy}{"58}
\usepackage[T2A,T1]{fontenc}
\usepackage[OT2,T1]{fontenc}

\newtheorem{theorem}{Theorem}[section]
\newtheorem{lem}[theorem]{Lemma}

\newtheorem{prop}[theorem]{Proposition}
\newtheorem{ass}[theorem]{Assumption}

\newtheorem{defn}[theorem]{Definition}
\newtheorem{question}[theorem]{Question}

\newcommand{\cF}{\mathcal{F}}

\newcommand{\cX}{\mathfrak{X}}
\newcommand{\cY}{\mathfrak{Y}}

\newcommand{\cH}{\mathcal{H}}

\newcommand{\Z}{\mathbb{Z}}

\newcommand{\Q}{\mathbb{Q}}

\newcommand{\cL}{\mathcal{L}}

\newcommand{\cO}{\mathcal{O}}

\newcommand{\op}[1]{\operatorname{#1}}

\numberwithin{equation}{section}

\begin{document}

\title[On the corank of the fine Selmer group of an elliptic curve]{On the corank of the fine Selmer group of an elliptic curve over a $\Z_p$-extension}

\author[A.~Ray]{Anwesh Ray}
\address[A.~Ray]{Department of Mathematics\\
University of British Columbia\\
Vancouver BC, Canada V6T 1Z2}
\email{anweshray@math.ubc.ca}

\maketitle

\begin{abstract}
Let $p$ be an odd prime and $F_\infty$ be a $\Z_p$-extension of a number field $F$. Given an elliptic curve $E$ over $F$, we study the structure of the \emph{fine Selmer group} over $F_\infty$. It is shown that under certain conditions, the fine Selmer group is a cofinitely generated module over $\Z_p$ and furthermore, we obtain an upper bound for its corank (i.e., the $\lambda$-invariant), in terms of various local and global invariants.
\end{abstract}

\section{Introduction}
\par The variation of arithmetic objects in families of number fields of interest is a central theme in number theory. Iwasawa studied the variation of class numbers of number fields in certain infinite towers of number field extensions. In greater detail, let $p$ be a prime number and $F$ be a number field. Set $\Z_p$ to denote the ring of $p$-adic integers, defined to be the inverse limit $\varprojlim_n \Z/p^n\Z$. A \emph{$\Z_p$-tower} consists of a family of Galois extensions $\{F_n\}$ of $F$, satisfying the containments
\[F=F_0\subset F_1\subset F_2\subset \dots \subset F_n\subset F_{n+1}\subset \dots,\] and such that $\op{Gal}(F_n/F)$ is isomorphic to $\Z/p^n\Z$. The infinite Galois extension $F_\infty:=\cup_{n\geq 1} F_n$ is referred to a $\Z_p$-extension, since the Galois group $\op{Gal}(F_\infty/F)$ is isomorphic to $\Z_p$. Let $e_n$ be the highest power of $p$ dividing the class number of $F_n$. Iwasawa showed that there are integers $\mu, \lambda\in \Z_{\geq 0}$ and $\nu\in \Z$ such that for all sufficiently large values of $n$, we have that $e_n=p^n \mu+n \lambda+\nu$ (cf. \cite{iwasawa1973zl} or \cite[Chapter 13]{washington1997introduction}). Let $\mu_{p^\infty}$ be the $p$-primary roots of unity in $\bar{F}$; the \emph{cyclotomic $\Z_p$-extension of $F$} is the unique $\Z_p$-extension of $F$ contained in $F(\mu_{p^\infty})$. It is conjectured by Iwasawa that for the cyclotomic $\Z_p$-extension of $F$, the $\mu$-invariant vanishes. In other words, Iwasawa conjectured that over the cyclotomic $\Z_p$-extension of $F$, the $p$-primary part of the Hilbert class group is cofinitely generated as a $\Z_p$-module. The corank of this $\Z_p$-module is the $\lambda$-invariant. 
\par Mazur (cf. \cite{mazur1972rational}) initiated the Iwasawa theory of abelian varieties with good ordinary reduction. The original motivation was to study the growth of Selmer groups, Tate-Shafarevich groups and Mordell--Weil ranks of an abelian variety in a $\Z_p$-tower of number fields. The Selmer groups in question are defined by local conditions at each prime, defined in terms of certain local Kummer maps. These Selmer groups capture various arithmetic properties of the elliptic curves or abelian varieties in question. A closer analogue with the Hilbert class groups holds for the \emph{fine Selmer group}, which is the subgroup defined by \emph{strict} local conditions at all primes (cf. Definition \ref{fine selmer def}). The fine Selmer group appeared in various guises in work of Perrin-Riou (cf. \cite{perrin1993fonctions, perrin1995fonctions}), Rubin (cf. \cite{rubin2000euler}) and Kato (cf. \cite{kato2004p}). Their algebraic properties were systematically studied by Coates and Sujatha in \cite{coates2005fine}, with an emphasis on the vanishing of the $\mu$-invariant. Indeed, it is conjectured that the fine Selmer group of an elliptic curve over the cyclotomic $\Z_p$-extension of a number field is a finitely generated $\Z_p$-module (cf. Conjecture A in \emph{loc. cit.}). In recent work (cf. \cite{deo2022mu}), it it shown that this conjecture is indeed satisfied over the cyclotomic $\Z_p$-extension, for various Galois representations of interest, provided certain additional conditions are satisfied.
\par When the $\mu$-invariant is known to vanish, the $\Z_p$-corank of the fine Selmer group is equal to the $\lambda$-invariant. The main focus of this paper is to shed light on the $\lambda$-invariant of the fine Selmer group of an elliptic curve over a general $\Z_p$-extension. For the classical Selmer groups of elliptic curve with good ordinary reduction, the $\lambda$-invariant over the cyclotomic $\Z_p$-extension can be computed since there are efficient algorithms to compute the $p$-adic L-function \cite{stein2007sage}. This calculation is possible since it the \emph{main conjecture} is known to hold in this context, thanks to the seminal work of Skinner and Urban \cite{skinner2014iwasawa}. The $\Z_p$-corank of the fine Selmer group is shrouded in mystery since there are not many satisfactory computational tools to obtain a precise description of the characteristic ideal except for certain trivial cases. It should be noted here that one possible techniques used in obtaining some insight into the $\lambda$-invariant is via the \emph{Euler characteristic formula} for the leading term of the characteristic series (cf. \cite[section 6]{wuthrich2007iwasawa}). If the formula for the leading term is a $p$-adic unit, then, under further hypotheses, the $\lambda$-invariant may be exactly computed to be equal to $\op{max}\{0, r-1\}$, where $r$ is the Mordell-Weil rank of the elliptic curve in question. Such ideas are used in studying distribution questions for the Iwasawa invariants of the fine Selmer groups of elliptic curves (cf. \cite{ray2021arithmetic}). However, this technique does not provide any upper bound on the $\lambda$-invariant, when the leading term is not a $p$-adic unit. The upper bounds obtained in this paper can be expressed in terms of various local and global invariants. Theorem \ref{th 4.5} describes and upper bound which holds in generality. The formula is better understood in a special case in which the global invariants in our formula are shown to vanish, see Theorem \ref{main thm}. We provide a concrete example to better clarify the main results of the paper.

%\par \textbf{Aknowledgments:}

\section{Preliminary Notions}
\par In this section, we introduce relevant notation and recall the definition and basic properties of fine Selmer groups associated to elliptic curves. Let $F$ be a number field. Throughout, $p$ will be an odd prime number. Fix an algebraic closure $\bar{F}$ of $F$. For each prime $v$ of $F$, set $F_v$ to denote the completion of $F$ at $v$, and let $\bar{F}_v$ be a choice of algebraic closure of $F_v$. Fix an embedding $\iota_v:\bar{F}\hookrightarrow \bar{F}_v$ for each prime $v$ of $F$. Given a field $\cF$ contained in $\bar{F}$, let $\op{G}_{\cF}$ denote the absolute Galois group $\op{Gal}(\bar{F}/\cF)$. 
\par Let $F_\infty$ be a \emph{$\Z_p$-extension} of $F$, i.e., $F_\infty$ is an infinite Galois extension of $F$ for which $\op{Gal}(F_\infty/F)$ is topologically isomorphic to $\Z_p$. For $n\in \Z_{\geq 1}$, we set $F_n$ to be the extension of $F$ such that $[F_n:F]=p^n$. Note that $F_n$ is a Galois extension of $F$ and that $\op{Gal}(F_n/F)$ is isomorphic to $\Z/p^n\Z$. Let $E_{/F}$ be an elliptic curve with good reduction at all primes $v|p$ of $F$. Let $S$ be the set of primes $v$ of $F$ such that either $v|p$ or such that $E$ has bad reduction at $v$. Let $S_p$ be the set of primes of $F$ that lie above $p$. Given a set of primes $\Sigma$ of $F$ and an algebraic extension $\cF/F$, let $\Sigma(\cF)$ be the set of primes of $\cF$ that lie above the set $\Sigma$. Given a prime $v$ of $F$, set $v(\cF)$ to be the set of primes of $\cF$ that lie above $v$.

\par Let $\cF$ be an extension of $F$ contained in $\bar{F}$. For a prime $w$ in $\cF$, let $\cF_w$ be the completion of $\cF$ at $w$. For ease of notation, $H^i(\cF_w, \cdot)$ (resp. $H^i(\cF_w, \cdot)$) will be used as a shorthand for $H^i(\op{G}_{\cF}, \cdot)$ (resp. $H^i(\op{G}_{\cF_w}, \cdot)$). Given a prime $v$ of $F$, we set $\cH_v\left(\cF, E[p^\infty]\right)$ to denote the following (possibly infinite) product
\[\cH_v\left(\cF, E[p^\infty]\right):=\prod_{w\in v(\cF)} H^1(\cF_w, E[p^\infty]).\] Note that all primes $v\notin S_p$ are unramified in any $\Z_p$-extension $F_\infty$. Therefore, $F_\infty$ is contained in $F_S$. Let $F_S$ denote the maximal extension of $F$ contained in $\bar{F}$ in which all the primes $v\notin S$ are unramified. Assume that $\cF$ is contained in $F_S$.
The $p$-primary fine Selmer group of $E$ over $\cF$ is defined as follows
\begin{equation}\label{fine selmer def}R\left(E/\cF\right):=\op{ker}\left\{H^1(F_S/\cF, E[p^\infty])\rightarrow \bigoplus_{v\in S} \cH_v\left(\cF,E[p^\infty]\right)\right\},\end{equation} where the map in question is the natural map induced by restriction to $\cF_w$ as $w$ ranges over the set $S(\cF)$.

We set $\Gamma$ to denote the Galois group $\Gamma:=\op{Gal}\left(F_\infty/F\right)$. The Iwasawa algebra $\Lambda$ is the completed group algebra $\Lambda=\varprojlim_n \Z_p[\Gamma/\Gamma^{p^n}]$. Let $\gamma$ be a topological generator of $\Gamma$, and set $T$ to denote $(\gamma-1)$. We identify the Iwasawa algebra $\Lambda$ with the formal power series ring $\Z_p\llbracket T\rrbracket$. A polynomial $f(T)\in \Lambda$ is said to be \emph{distinguished} if it is a nonconstant monic polynomial and all of its non-leading coefficients are divisible by $p$.

\par We introduce the Iwasawa invariants associated to a finitely generated module over $\Lambda$. A map of $\Lambda$-modules $M_1\rightarrow M_2$ is said to be a \emph{pseudo-isomorphism} if its kernel and cokernel are both finite. Let $M$ be a finitely generated module over $\Lambda$. According to the structure theorem of $\Lambda$-modules \cite[Chapter 13]{washington1997introduction}, there is a pseudoisomorphism of the form
\[M\longrightarrow \Lambda^r\oplus \left(\bigoplus_{i=1}^s \Lambda/(p^{\mu_i})\right)\oplus \left(\bigoplus_{j=1}^t \Lambda/\left(f_j(T)^{\lambda_j}\right)\right),\] where $r\in \Z_{\geq 0}$ is the rank, $\mu_i$ are positive integers and $f_j(T)$ are irreducible distinguished polynomials. The $\mu$ and $\lambda$-invariants are defined as follows
\[\mu(M):=\sum_{i=1}^s \mu_i, \text{ and }\lambda(M):=\sum_{j=1}^t \lambda_j\op{deg}f_j,\]where we set $\mu(M)=0$ (resp. $\lambda(M)=0$) if $s=0$ (resp. $t=0$).
Denote the $\mu$ (resp. $\lambda$-invariant) of $R(E/F_\infty)^\vee$ by $\mu(E/F_\infty)$ (resp. $\lambda(E/F_\infty)$). We note that when $F_\infty/F$ is the cyclotomic $\Z_p$-extension, it is conjectured that $R(E/F_\infty)^\vee$ is a torsion $\Lambda$-module and $\mu(E/F_\infty)=0$, cf. \cite[Conjecture A]{coates2005fine}.

\par We are interested in the following question.
\begin{question}\label{mu=0 conjecture}
 Under what conditions can it be shown that the dual fine Selmer group $R(E/F_\infty)^\vee$ is a finitely generated torsion $\Lambda$-module with $\mu(E/F_\infty)=0$? Furthermore, assuming that these conditions are satisfied, how well can one characterize the $\lambda$-invariant $\lambda(E/F_\infty)$?
\end{question}

That $R(E/F_\infty)^\vee$ is a finitely generated as a $\Lambda$-module is well known, and follows from a standard argument that involves an application of Nakayama's lemma. 
\begin{lem}\label{torsion mu=0 lemma}
With respect to notation above, the following conditions are equivalent. 
\begin{enumerate}
    \item $R(E/F_\infty)^\vee$ is a torsion $\Lambda$-module with $\mu(E/F_\infty)=0$, 
    \item $R(E/F_\infty)[p]$ is finite.
\end{enumerate}  
Furthermore, if the above conditions are satisfied, then,
\[\lambda(E/F_\infty)\leq \op{dim}_{\Z/p\Z} \left(R(E/F_\infty)[p]\right).\]
\end{lem}
\begin{proof}
Let $M$ be a finitely generated $\Lambda$-module. It is an easy consequence of the structure theory that $M$ is a torsion $\Lambda$-module with $\mu(M)=0$ if and only if $M$ is finitely generated as a $\Z_p$-module. Let $M$ be the dual fine Selmer group $R(E/F_\infty)^\vee$, we find therefore that $M$ is a torsion $\Lambda$-module with $\mu(M)=0$ if and only if $M/pM$ is finite. Note that the dual of $R(E/F_\infty)[p]$ is $M/pM$, and hence the two conditions above are equivalent. Furthermore, $\lambda(M)$ is equal to the $\Z_p$-rank of $M$, and hence, $\lambda(M)\leq \op{dim}_{\Z/ p\Z} M/pM$. Therefore, we find that $\lambda(E/F_\infty)\leq \op{dim}_{\Z / p\Z} \left(R(E/F_\infty)[p]\right)$.
\end{proof}

Given a $\Lambda$-module $M$, we say that $M$ is cofinitely generated (resp. cotorsion) as a $\Lambda$-module if the Pontryagin dual $M^\vee:=\op{Hom}\left(M, \Q_p/\Z_p\right)$ is finitely generated (resp. torsion) as a $\Lambda$-module.
\section{The residual fine Selmer group}
\par Recall that $F$ is a number field and $E$ is an elliptic curve defined over $F$, and $p$ is an odd prime. Denote by $E[p]$ the $p$-torsion subgroup of $E(\bar{F})$ and note that as an abelian group, $E[p]$ is isomorphic to $\left(\Z/ p \Z\right)^2$. We define an analogue of the fine Selmer group associated with the residual representation on $E[p]$.

\par Recall that $S$ consists of the set of primes $v\nmid p$ at which $E$ has bad reduction and the primes $v\mid p$. For $v\in S$, set $\cH_v(F_\infty, E[p])$ to be the (possibly infinite) product $\prod_{w\in v(F_\infty)} H^1(F_{\infty, w}, E[p])$. The \emph{residual fine Selmer group} is defined as follows
\[R(E[p]/F_\infty):=\op{ker}\left\{H^1(F_S/F_\infty, E[p])\xrightarrow{\Phi} \bigoplus_{v\in S} \cH_v\left(F_\infty,E[p]\right)\right\}.\] The map $\Phi$ is the natural map obtained by restriction to $F_{\infty, w}$ as $w$ ranges over the primes in $S(F_\infty)$. Consider the Kummer sequence 
\[0\rightarrow E[p]\rightarrow E[p^\infty]\rightarrow E[p^\infty]\rightarrow 0.\]Given a prime $v\in S$, let $h_v$ be the map 
\[h_v: \cH_v(F_\infty, E[p])\rightarrow \cH_v(F_\infty, E[p^\infty])[p]\] which is the product of the natural maps induced from the Kummer sequence
\[h_w: H^1(F_{\infty, w}, E[p])\rightarrow H^1(F_{\infty, w}, E[p^\infty])[p],\] as $w$ ranges over $v(F_\infty)$.
Let $h$ be the direct sum of the maps $h_v$, as $v$ ranges over $S$
\[h:\bigoplus_{v\in S}\cH_v(F_\infty, E[p]) \longrightarrow \bigoplus_{v\in S}\cH_v(F_\infty, E[p^\infty])[p].\]
From the Kummer sequence, we have the following diagram
\begin{equation}\label{fdiagram}
\begin{tikzcd}[column sep = small, row sep = large]
0\arrow{r} & R(E[p]/F_\infty) \arrow{r}\arrow{d}{\Psi} & H^1(F_S/F_\infty, E[p])\arrow{r} \arrow{d}{g} & \operatorname{im}(\Phi)\arrow{r} \arrow{d}{h'} & 0\\
0\arrow{r} & R(E/F_\infty)[p] \arrow{r} & H^1(F_S/F_\infty, E[p^{\infty}])[p] \arrow{r}  &\bigoplus_{v\in S} \mathcal{H}_v(F_\infty, E[p^{\infty}])[p].
\end{tikzcd}
\end{equation}
In the above diagram, the map $g$ is surjective with kernel isomorphic to $\frac{H^0(F_\infty, E[p^\infty])}{p H^0(F_\infty, E[p^\infty])}$. The vertical map $h'$ is the restriction of the map $h$ to $\op{im}(\Phi)$. By an application of the Snake lemma, we have the following exact sequence
\begin{equation}\label{exact}0\rightarrow \op{ker} \Psi\rightarrow \op{ker} g\rightarrow \op{ker} h'\rightarrow \op{cok} \Psi\rightarrow 0. \end{equation}
Note that for $v\in S$, the kernel of $h_v$ is the product of kernels of $h_w$ as $w$ ranges over the primes of $F_\infty$ that lie above $v$. 
\begin{prop}\label{prop 3.1}
With respect to notation above, assume that the following conditions are satisfied
\begin{enumerate}
    \item $\op{ker}(h)$ is finite, 
    \item $R(E[p]/F_\infty)$ is finite.
\end{enumerate}Then, the fine Selmer group $R(E/F_\infty)$ is a cotorsion $\Lambda$-module with $\mu(E/F_\infty)=0$. Furthermore, we have that 
\[\lambda(E/F_\infty)\leq \op{dim}_{\Z/p\Z}\left(R(E[p]/F_\infty)\right)+\op{dim}_{\Z/p\Z}\left(\op{ker}(h)\right).\]
\end{prop}
\begin{proof}
First, we show that $R(E/F_\infty)$ is a cotorsion $\Lambda$-module and $\mu(E/F_\infty)=0$. According to Lemma \ref{torsion mu=0 lemma}, $R(E/F_\infty)$ is a cotorsion $\Lambda$-module with $\mu=0$ if and only if $R(E/F_\infty)[p]$ is finite. Consider the map 
\[\Psi: R(E[p]/F_\infty)\rightarrow R(E/F_\infty)[p],\] (cf. \eqref{fdiagram}) and we observe that it suffices to show that $\op{ker}\Psi$ and $\op{cok}\Psi$ are finite. Referring to the exact sequence \eqref{exact}, we deduce that $\op{ker}\Psi$ is finite since $\op{ker}(g)$ is finite. Since it is assumed that $\op{ker}(h)$ is finite, it follows that $\op{cok}\Psi$ is finite as well. Therefore, we have shown that $R(E/F_\infty)[p]$ is finite, and therefore deduce that $R(E/F_\infty)$ is a cotorsion $\Lambda$-module and $\mu(E/F_\infty)=0$.
\par Note that it also follows from Lemma \ref{torsion mu=0 lemma} that $\lambda(E/F_\infty)\leq \op{dim}_{\Z/p\Z} \left(R(E/F_\infty)[p]\right)$. Hence, we arrive at the bounds
\[\begin{split}\lambda(E/F_\infty) &\leq \op{dim}_{\Z/p\Z} \left(R(E/F_\infty)[p]\right) \\ &\leq \op{dim}_{\Z/p\Z} \left(R(E[p]/F_\infty)\right) +\op{dim}_{\Z/p\Z}\left(\op{cok}\Psi\right) \\ &\leq \op{dim}_{\Z/p\Z} \left(R(E[p]/F_\infty)\right) +\op{dim}_{\Z/p\Z}\left(\op{ker}(h)\right),\end{split}\]
and thus, this completes the proof.
\end{proof}

\begin{defn}Let $S_0$ be the set of primes $v\in S$ at which \emph{both} of the following conditions are satisfied
\begin{enumerate}
\item $E$ has bad reduction at $v$.
\item If $p\geq 5$ and $\mu_p$ is contained in $F_v$, then, $E$ has split multiplicative reduction at $v$.
\end{enumerate}
\end{defn}

\begin{ass}\label{main assumption}
Assume that each prime $v\in S_0\cup S_p$ is finitely decomposed in $F_\infty$.
\end{ass}
Since all primes of $F$ are finitely decomposed in its cyclotomic $\Z_p$-extension, the above assumption is always satisfies when $F_\infty$ is the cyclotomic $\Z_p$-extension of $F$.
Note that since $E$ is stipulated to have good reduction at all primes $v\in S$, we find that $S_0$ is a subset of $S\backslash S_p$. Given $v\in S_p$, set $\delta_v:=\begin{cases} & 2\text{ if }E(F_v)[p]\neq 0;\\
& 0\text{ if }E(F_v)[p]= 0.
\end{cases}$ For any prime $v\in S\cup S_p$, let $g_v$ be the number of primes $w\in v(F_\infty)$.

\begin{lem}\label{lemma 3.4}
It follows from Assumption \ref{main assumption} that $\op{ker}(h)$ is finite, and moreover, \[\op{dim}\left(\op{ker}(h)\right)\leq \sum_{v\in S_0} 2g_v+\sum_{v\in S_p}\delta_v g_v.\]
\end{lem}

\begin{proof}
Let $v\in S$ and let $w$ be a prime of $F_\infty$ that lies above $v$. From the Kummer sequence, we find that the kernel of $h_w$ is identified with $E(F_{\infty,w})[p^\infty]\otimes \Z/p\Z$. Note that in particular, the dimension of $\op{ker}h_w$ is at most $2$. There are three cases to consider.
\begin{enumerate}
\item If $v\in S_0$, then, $\op{dim} \left(\op{ker}(h_w)\right)$ is at most $2$. Thus, $\op{dim}\left(\op{ker} (h_v)\right)$ is at most $2g_v$.
\item If $v\in S\backslash (S_0\cup S_p)$, then, $\mu_p$ is contained in $F_v$ and $E$ has either nonsplit multiplicative reduction or additive reduction at $v$. Since $v\nmid p$, it is unramified in any $\Z_p$-extension, in particular, unramified in $F_\infty$. Thus, the elliptic curve $E$ has either non-split multiplicative reduction or additive reduction at every prime $w$ of $F_\infty$ that lies above $v$. Then, it follows from \cite[Proposition 5.1. (iii)]{hachimori1999analogue} that $E(F_{\infty,w})[p^\infty]=0$ for all primes $w\in v(F_\infty)$. As a result, the kernel of $h_v$ is $0$ in this case.
\item Finally, consider the case when $v\in S_p$. There are two subcases to consider.
\begin{enumerate}
\item First, assume that $E(F_v)[p]=0$. Then, since $F_{\infty, w}/F_v$ is a pro-$p$ extension, it follows that $E(F_{\infty,w})[p^\infty]=0$ (cf. \cite[Proposition 1.6.12]{neukirch2013cohomology}). Therefore, the kernel of $h_v$ is $0$ in this case.
\item On the other hand, if $E(F_v)[p]\neq 0$, then, we do still have the bound $\op{dim}\left(\op{ker}h_w\right)\leq 2$ for each prime $w$ of $F_\infty$ that lies above $w$. 
\end{enumerate}
From the above case decomposition, we find that $\op{dim}h\leq \sum_{v\in S_0} 2g_v+\sum_{v\in S_p}\delta_v g_v$. Since $g_v$ is assumed to be finite for each $v\in S_0\cup S_p$, the sum is finite.
\end{enumerate}
\end{proof}
%\begin{prop}
%Let $r$ be the corank of $R(E/F_\infty)$ as a $\Lambda$-module, and $\mu=\mu(E/F_\infty)$ its $\mu$-invariant. Suppose that Assumption \ref{main assumption} is satisfied. Then, $r\leq 2$ and if $r=2$, then, $\mu=0$.
%\end{prop}

\section{Main results and their proofs}
\par In this section, we state and prove the main results of the paper. We illustrate the bounds obtained via a concrete example at the end of the section. 

\par Given an algebraic extension $\cF$, let $H(\cF)$ (resp. $A(\cF)$) be the maximal abelian unramified extension of $F$ such that $\op{Gal}(H(\cF)/\cF)$ (resp. $\op{Gal}(A(\cF)/\cF)$) is an elementary $p$-group (resp. pro-$p$ group). Let $\cF(E[p])$ be the extension of $\cF$ generated by $E[p]$. In other words, the field $\cF(E[p])$ is the field extension of $\cF$ which is fixed by the kernel of $\bar{\rho}:\op{G}_{\cF}\rightarrow \op{Aut}(E[p])$. Let $K_n$ be the extension $F_n(E[p])$ and set $L_n:=H(K_n)$. Set $K_\infty$ to denote $F_\infty(E[p])$ and set $L_\infty$ (resp. $\cL_\infty$) to denote the extension $H(K_\infty)$ (resp. $A(K_\infty)$). We let $X:=\op{Gal}(L_\infty/K_\infty)$ and $\cX:=\op{Gal}(\cL_\infty/K_\infty)$. We set $K$ to denote the extension $F(F[p])$, and note that $K_\infty$ is a $\Z_p$-extension of $K$. Let $\Lambda=\Z_p\llbracket \op{Gal}(K_\infty/K)\rrbracket$ denote the associated Iwasawa algebra, and note that $\cX$ is a module over $\Lambda$. It is in fact known that $\cX$ is finitely generated module over $\Lambda$ (cf. \cite[Chapter 13]{washington1997introduction}). By construction, $X=\cX/p\cX$ is a module over $\Omega:=\Lambda/p$. \begin{defn}\label{def of Y}Denote by $H^S(K_\infty)$ the maximal extension of $K_\infty$, which is contained in $H(K_\infty)$, in which the primes of $S(K_\infty)$ are completely split. Set $\cY$ to denote $\op{Gal}(H^S(K_\infty)/K_\infty)$, and set $Y:=\cY/p\cY$.
\end{defn}Set $G$ to denote the Galois group $\op{Gal}(K_\infty/F_\infty)$. From the inflation-restriction sequence, we have the following exact sequence

\[0\rightarrow H^1(G, E(K_\infty)[p])\xrightarrow{\op{inf}} H^1(F_S/F_\infty, E[p])\xrightarrow{\op{res}} H^1(K_S/K_\infty, E[p]).\]
Note that the action of $\op{Gal}(K_S/K_\infty)$ on $E[p]$ is the trivial action, and therefore, we find that \[H^1(K_S/K_\infty, E[p])=\op{Hom}\left(\op{Gal}(K_S/K_\infty),\Z/p\Z\right)^2.\]
The residual fine Selmer group $R(E[p]/F_\infty)$ is a subgroup of $H^1(F_S/F_\infty, E[p])$. Let $\mathcal{Z}$ denote $\op{inf}^{-1}(R(E[p]/F_\infty))$ and set $\mathcal{Z}_1:=\op{res}(Y)$. Note that $\mathcal{Z}_1$ can be identified with a subgroup of $\op{Hom}(Y, \Z/p\Z)^2$. In this way, we have an exact sequence
\begin{equation}\label{main exact sequence}0\rightarrow \mathcal{Z}\rightarrow R(E[p]/F_\infty)\rightarrow \op{Hom}(Y, \Z/p\Z)^2,\end{equation} where the image of the rightmost map is $\mathcal{Z}_1$. Since $G$ is finite, it is clear that $\mathcal{Z}$ is finite as well.
The following assumption is a special case of Iwasawa's classical $\mu=0$ conjecture. 

\begin{ass}\label{class tower ass}
Assume that $\cY$ is torsion $\Lambda$-module and that $\mu(\cY)=0$. Equivalently, assume that $\cY$ is finitely generated as a $\Z_p$-module.
\end{ass}

\begin{lem}\label{lemma 4.2}
With respect to notation above, the following conditions are equivalent.
\begin{enumerate}
\item Assumption \ref{class tower ass} is satisfied,
\item $Y$ is finite.
\end{enumerate}
Furthermore, if the above equivalent conditions are satisfied, then, $\lambda(\cY)\leq \op{rank}_{\Z_p} \cY$.
\end{lem}
\begin{proof}
It is an easy consequence of the structure theory of finitely generated Iwasawa modules that Assumption \ref{class tower ass} is satisfied if and only if $\cY$ is finitely generated as a $\Z_p$-module and that $\lambda(\cY)=\op{rank}_{\Z_p}\cY$. Note that \[\op{rank}_{\Z_p}\cY\leq \op{dim}_{\Z/p\Z}\left(\cY/p\cY\right)=\op{dim}_{\Z/p\Z}Y,\] and the result follows.
\end{proof}
In many situations, it is shown that certain Iwasawa modules that arise from class groups do not contain any non-zero finite $\Lambda$-submodules. This property, if it holds, implies that the dimension of $Y$ would equal the $\lambda$-invariant of $\cY$, as the following result shows.
\begin{prop}\label{prop 4.2}
With respect to notation above, suppose that Assumption \ref{class tower ass} holds. Then the following conditions are equivalent
\begin{enumerate}
\item\label{c1 prop 4.2} $\cY$ does not contain any non-zero finite $\Lambda$-submodules, 
\item\label{c2 prop 4.2} $\lambda(\cY)=\op{dim}_{\Z/p\Z} Y$.
\end{enumerate}
\end{prop}
\begin{proof}
It follows from Lemma \ref{lemma 4.2} that $\op{rank}_{\Z_p}\cY$ is finite and 
\[\cY=\Z_p^{\lambda(\cY)}\oplus A,\] where $A$ is a finite abelian $p$-group. Note that $A$ is the $p$-primary torsion subgroup of $\cY$, hence, $A$ is a finite $\Lambda$-submodule of $\cY$. Therefore, condition \eqref{c1 prop 4.2} is equivalent to the vanishing of $A$. On the other hand, $\op{dim}_{\Z/p\Z}(Y)=\lambda(\cY)+\op{dim}_{\Z/p\Z}(A)$. Therefore, \eqref{c1 prop 4.2} and \eqref{c2 prop 4.2} are equivalent.
\end{proof}
\begin{prop}\label{prop 4.4}
    With respect to notation above, suppose that Assumption \ref{class tower ass} holds.
    Then, $R(E[p]/F_\infty)$ is finite and
    \[\op{dim} R(E[p]/F_\infty)\leq 2\op{dim}(Y)+\op{dim}(\mathcal{Z}).\]
\end{prop}
\begin{proof}
It follows from Lemma \ref{lemma 4.2} that $Y$ is finite. The result follows from the exact sequence \eqref{main exact sequence}. 
\end{proof}
\begin{theorem}\label{th 4.5}
Let $E$ be an elliptic curve over a number field $F$ and let $p$ be an odd prime such that $E$ has good reduction at all primes that lie above $p$. Let $F_\infty/F$ be a $\Z_p$-extension of $F$. Suppose that Assumptions \ref{main assumption} and \ref{class tower ass} hold. Let $Y$ be as in Definition \ref{def of Y}. Then, the following assertions hold
\begin{enumerate}
    \item\label{p1 of th 4.5} $R(E/F_\infty)$ is a cotorsion module over $\Lambda$ with $\mu(E/F_\infty)=0$,
    \item\label{p2 of th 4.5} $Y$ is finite and the following bound on the $\lambda$-invariant holds
    \[\lambda(E/F_\infty)\leq 2\op{dim}(Y)+\op{dim}(\mathcal{Z})+\sum_{v\in S_0} 2g_v+\sum_{v\in S_p} \delta_v g_v.\]
\end{enumerate}

\end{theorem}

\begin{proof}
Recall that Proposition \ref{prop 3.1} asserts that if
\begin{enumerate}
    \item $\op{ker}(h)$ is finite, 
    \item $R(E[p]/F_\infty)$ is finite,
\end{enumerate}then, $R(E/F_\infty)$ is a cotorsion $\Lambda$-module with $\mu(E/F_\infty)=0$. Furthermore, we have that 
\[\lambda(E/F_\infty)\leq \op{dim}R(E[p]/F_\infty)+\op{dim}(\op{ker}(h)).\]
Lemma \ref{lemma 3.4} states that $\op{ker}(h)$ is finite and $\op{dim}(\op{ker}(h))$ is bounded above by $\sum_{v\in S_0} 2g_v+\sum_{v\in S_p} \delta_v g_v$. The result follows therefore from Proposition \ref{prop 4.4}.
\end{proof}

%As a Corollary to the above, we obtain an upper bound for the Mordell-Weil rank of an elliptic curve $E_{/\Q}$ with reduction at $p$.

%\begin{cor}
%Let $E$ be an elliptic curve over $\Q$ and let $p$ be an odd prime such that $E$ has good reduction at $p$. Let $\Q_\infty/\Q$ be the cyclotomic $\Z_p$-extension of $\Q$. With respect to notation above, assume that Assumption \ref{main assumption} and \ref{class tower ass} are satisfied. Then, the following bound on the Mordell-Weil rank of $E(\Q)$ holds
%\[\op{rank}_{\Z}E(\Q)\leq 1+2\op{dim}(Y)+\op{dim}(\mathcal{Z})+\sum_{v\in S_0} 2g_v+\sum_{v\in S_p} \delta_v g_v.\]
%\end{cor}

We now show that $Y$ and $\mathcal{Z}$ can be shown to be trivial provided further conditions are satisfied. Let $\bar{\rho}:\op{G}_F\rightarrow \op{GL}_2(\Z/p\Z)$ be the \emph{residual representation} on $E[p]$.

\begin{lem}
Suppose that the image of $\bar{\rho}$ is a non-solvable subgroup of $\op{GL}_2(\Z/p\Z)$. Then, $\bar{\rho}_{|\op{G}_{F_\infty}}$ is irreducible.
\end{lem}
\begin{proof}
Note that if $\bar{\rho}_{|\op{G}_{F_\infty}}$ is reducible, then, its image is solvable. Since $F_\infty$ is an abelian extension of $F$, we find that the image of $\bar{\rho}$ must also be solvable.
\end{proof}

\begin{lem}\label{lemma 4.7} The following assertions are satisfied.
\begin{enumerate}
\item\label{p1 of lemma 4.7} Suppose that \emph{any one} (or both) of the following conditions for $\bar{\rho}$ is satisfied
\begin{enumerate}
    \item $\bar{\rho}$ is irreducible when restricted to $\op{G}_{F_\infty}$,
    \item the image of $\bar{\rho}$ has cardinality coprime to $p$.
\end{enumerate} Then, we have that $\mathcal{Z}=0$.

\item\label{p2 of lemma 4.7} Suppose that both of the following conditions are satisifed
\begin{enumerate}
    \item there is only one prime $v$ of $K$ that lies above $p$ which ramifies in $K_\infty$ and $v$ is totally ramified in $K_\infty$,
    \item $p$ does not divide the class number of $K$.
\end{enumerate}Then, it follows that $Y=0$.
\end{enumerate}

\end{lem}
\begin{proof}
\par We first prove part \eqref{p1 of lemma 4.7}. Since $\bar{\rho}$ is assumed to satisfy one (or both) of two conditions, we consider the conditions one at a time. First, assume that $\bar{\rho}$ is irreducible when restricted to $\op{G}_{F_\infty}$. Then, the vanishing of $\mathcal{Z}$ follows from \cite[Lemma 2.2]{prasad2021relating}. Next, assume that the image of $\bar{\rho}$ has cardinality coprime to $p$. Since $E[p]\simeq \left(\Z/p\Z\right)^2$ is a $p$-group, and $G$ has cardinality coprime to $p$, it follows that $H^1(G, E[p])=0$. As a consequence, we deduce that $\mathcal{Z}=0$.
\par Part \eqref{p2 of lemma 4.7} follows from \cite[Proposition 13.22]{washington1997introduction}.
\end{proof}
\begin{theorem}\label{main thm}
Let $E$ be an elliptic curve over a number field $F$ and let $p$ be an odd prime such that $E$ has good reduction at all primes in $S_p$. Let $F_\infty/F$ be a $\Z_p$-extension of $F$. With respect to notation above, assume that Assumption \ref{main assumption} is satisfied. 
Furthermore, suppose that the following conditions are satisfied
\begin{enumerate}
\item the image of $\bar{\rho}$ is non-solvable or has cardinality coprime to $p$, 
\item there is only one prime of $K$ which ramifies in $K_\infty$ and this prime is totally ramified.
\item Suppose that $A(K)=0$.
\end{enumerate}
Then, we find that $R(E/F_\infty)$ is a cotorsion module over $\Lambda$ with $\mu(E/F_\infty)=0$, and\[\lambda(E/F_\infty)\leq \sum_{v\in S_0} 2g_v+\sum_{v\in S_p} \delta_vg_v.\] In particular, if $S_0=\emptyset$ and $E(F_v)[p]=0$ for all primes $v\in S_p$, then, $\lambda(E/F_\infty)=0$ for all $\Z_p$-extensions $F_\infty/F$.
\end{theorem}
\begin{proof}
Note that Lemma \ref{lemma 4.7} implies that $Y=0$ and hence, Assumption \ref{class tower ass} is automatically satisfied. The result follows directly form Theorem \ref{th 4.5} and Lemma \ref{lemma 4.7}. 
\end{proof}

\par \textbf{An Example:} We illustrate the above results through an example. Let us consider the elliptic curve $E$ with Cremona label \href{https://www.lmfdb.org/EllipticCurve/Q/11/a/2}{11a2} and let $p=5$. The residual representation $\bar{\rho}:\op{G}_{\Q}\rightarrow \op{GL}_2(\Z/p\Z)$ is a direct sum of two characters $\bar{\rho}=\bar{\chi}\oplus 1$, where $\bar{\chi}$ is a mod-$p$ cyclotomic character.
\begin{enumerate}
\item\label{example p1} First, consider the case when $F=\Q$ and $F_\infty=\Q_\infty$ is the cyclotomic $\Z_p$-extension of $\Q$. Note that $K=\Q(E[p])$ is the number field $\Q(\mu_p)$, and $K_\infty$ is $\Q(\mu_{p^\infty})$, the cyclotomic $\Z_p$-extension of $K$. The set of primes $S$ is equal to $\{11, p\}$ and $S_0=\{11\}$. All primes of $\Q$ are finitely decomposed in $\Q_\infty$, the cyclotomic $\Z_p$-extension of $\Q$, and hence, Assumption \ref{main assumption} is satisfied. We verify the conditions of Theorem \ref{main thm}. 
\begin{itemize}
\item Clearly the image of $\bar{\rho}$ has cardinality coprime to $p$.
\item Let $\eta_p$ be the prime of $K$ that lies above $p$. Note that $\eta_p$ is the principal ideal generated by $(1-e^{2\pi i/p})$ and that $p \cO_K=\eta_p^{p-1}$. The prime $\eta_p$ is totally ramified in $K_\infty=\Q(\mu_{p^\infty})$.
\item The prime $p=5$ is a regular prime, i.e., $p$ does not divide the class number of $K=\Q(\mu_p)$ (the smallest irregular prime is $37$).
\end{itemize}
We compute $g_{11}$, i.e, the number of primes of $\Q_\infty$ that lie above $11$. Since $5^2\nmid (11^4-1)$, it follows that $g_{11}=1$. On the other hand, $p$ is totally ramified in $\Q_\infty$, hence, $g_p=1$. Theorem \ref{main thm} then implies that $R(E/F_\infty)$ is a cotorsion module over $\Lambda$ with $\mu(E/\Q_\infty)=0$, and\[\lambda(E/\Q_\infty)\leq \sum_{v\in S_0} 2g_v+\sum_{v\in S_p} \delta_vg_v\leq 4.\]
\item Next, suppose that $F=\Q(\mu_p)$. Note that $K=F$ and that there are $p-1=4$ independent $\Z_p$-extensions of $F$. Let $F_\infty/F$ be any $\Z_p$-extension in which $11$ is finitely decomposed and in which $\eta_p$ is totally ramified. Thus in particular, Assumption \ref{main assumption} holds. Note that $K_\infty$ is equal to $F_\infty$. The same arguments as in part \eqref{example p1} above imply that the assumptions of Theorem \ref{main thm} are satisfied. Theorem \ref{main thm} then implies that $R(E/F_\infty)$ is a cotorsion module over $\Lambda$ with $\mu(E/F_\infty)=0$, and\[\lambda(E/F_\infty)\leq \sum_{v\in S_0} 2g_v+\sum_{v\in S_p} \delta_vg_v=\sum_{v|11}2g_{v}+\delta_{\eta_p}.\] When $F_\infty$ is taken to be the cyclotomic $\Z_p$-extension of $F$, then Assumption \ref{main assumption} is clearly satisfied. In this case, we find that $11$ is totally inert in $F_\infty$, and therefore, we find that $\lambda(E/F_\infty)\leq 4$.
\end{enumerate}

\bibliographystyle{alpha}
\bibliography{references}

\newcommand{\etalchar}[1]{$^{#1}$}
\begin{thebibliography}{{I}wa73}

\bibitem[CS05]{coates2005fine}
John Coates and Ramdorai Sujatha.
\newblock Fine selmer groups of elliptic curves over p-adic lie extensions.
\newblock {\em Mathematische Annalen}, 331(4):809--839, 2005.

\bibitem[DRS22]{deo2022mu}
Shaunak~V Deo, Anwesh Ray, and R~Sujatha.
\newblock On the $\mu$ equals zero conjecture for the fine selmer group in
  iwasawa theory.
\newblock {\em arXiv preprint arXiv:2202.09937}, 2022.

\bibitem[HM99]{hachimori1999analogue}
Yoshitaka Hachimori and Kazuo Matsuno.
\newblock An analogue of kida's formula for the selmer groups of elliptic
  curves.
\newblock {\em Journal of Algebraic Geometry}, 8(3):581--601, 1999.

\bibitem[{I}wa73]{iwasawa1973zl}
Kenkichi {I}wasawa.
\newblock On $\mathbb{Z}_\ell$-extensions of algebraic number fields.
\newblock {\em Annals of Mathematics}, pages 246--326, 1973.

\bibitem[K{\etalchar{+}}04]{kato2004p}
Kazuya Kato et~al.
\newblock p-adic hodge theory and values of zeta functions of modular forms.
\newblock {\em Ast{\'e}risque}, 295:117--290, 2004.

\bibitem[Maz72]{mazur1972rational}
Barry Mazur.
\newblock Rational points of abelian varieties with values in towers of number
  fields.
\newblock {\em Inventiones mathematicae}, 18(3):183--266, 1972.

\bibitem[NSW13]{neukirch2013cohomology}
J{\"u}rgen Neukirch, Alexander Schmidt, and Kay Wingberg.
\newblock {\em Cohomology of number fields}, volume 323.
\newblock Springer Science \& Business Media, 2013.

\bibitem[PR93]{perrin1993fonctions}
Bernadette Perrin-Riou.
\newblock Fonctions l $p$-adiques d’une courbe elliptique et points
  rationnels.
\newblock In {\em Annales de l'institut Fourier}, volume~43, pages 945--995,
  1993.

\bibitem[PR95]{perrin1995fonctions}
Bernadette Perrin-Riou.
\newblock {\em Fonctions L p-adiques des repr{\'e}sentations p-adiques}.
\newblock Number~94. Soci{\'e}t{\'e} math{\'e}matique de France, 1995.

\bibitem[PS21]{prasad2021relating}
Dipendra Prasad and Sudhanshu Shekhar.
\newblock Relating the tate--shafarevich group of an elliptic curve with the
  class group.
\newblock {\em Pacific Journal of Mathematics}, 312(1):203--218, 2021.

\bibitem[RS21]{ray2021arithmetic}
Anwesh Ray and R~Sujatha.
\newblock Arithmetic statistics for the fine selmer group in iwasawa theory.
\newblock {\em arXiv preprint arXiv:2112.13335}, 2021.

\bibitem[Rub00]{rubin2000euler}
Karl Rubin.
\newblock {\em Euler systems}.
\newblock Number 147. Princeton University Press, 2000.

\bibitem[Ste07]{stein2007sage}
William Stein.
\newblock Sage mathematics software.
\newblock {\em http://www. sagemath. org/}, 2007.

\bibitem[SU14]{skinner2014iwasawa}
Christopher Skinner and Eric Urban.
\newblock The iwasawa main conjectures for gl2.
\newblock {\em Inventiones mathematicae}, 195(1):1--277, 2014.

\bibitem[Was97]{washington1997introduction}
Lawrence~C Washington.
\newblock {\em Introduction to cyclotomic fields}, volume~83.
\newblock Springer Science \& Business Media, 1997.

\bibitem[Wut07]{wuthrich2007iwasawa}
Christian Wuthrich.
\newblock Iwasawa theory of the fine selmer group.
\newblock {\em Journal of Algebraic Geometry}, 16(1):83--108, 2007.

\end{thebibliography}
\end{document}